\documentclass[12pt]{amsart}

  \usepackage{graphicx}
  \usepackage{pictexwd,dcpic}
  \usepackage{latexsym,amssymb,epsfig,a4wide}

  \def\NN{{\mathbb N}}
  
  \def\ZZ{{\mathbb Z}}
  
  \def\kk{{\mathbf k}}
  
  \def\cC{{\mathcal C}}

  \def\sm{\smallsetminus}

  \theoremstyle{plain}
    \newtheorem{theorem}{Theorem}[section]
    \newtheorem{proposition}[theorem]{Proposition}
    
    \newtheorem{corollary}[theorem]{Corollary}
    \newtheorem{conjecture}[theorem]{Conjecture}
  \theoremstyle{definition}

  \numberwithin{equation}{section}

\begin{document}

\title[On the flag $f$-vector of a lattice with nontrivial homology]{On 
the flag $f$-vector of a graded lattice with nontrivial homology}

  \author{Christos~A.~Athanasiadis}
  \address{Department of Mathematics
          (Division of Algebra-Geometry) \\
          University of Athens \\
          Panepistimioupolis, Athens 15784 \\
          Hellas (Greece)}
\email{caath@math.uoa.gr}

\date{May 14, 2011}

  \begin{abstract}
    It is proved that the Boolean algebra of rank $n$ minimizes the flag
    $f$-vector among all graded lattices of rank $n$, whose proper part
    has nontrivial top-dimensional homology. The analogous statement for
    the flag $h$-vector is conjectured in the Cohen-Macaulay case.
  \end{abstract}

  \maketitle

  \section{Introduction}
  \label{sec:intro}

  Let $P$ be a finite graded poset of rank $n \ge 1$, having a
  minimum element $\hat{0}$, maximum element $\hat{1}$ and rank function
  $\rho: P \to \NN$ (we refer to \cite[Chapter 3]{StaEC1} for any undefined
  terminology on partially ordered sets). Given $S \subseteq [n-1] := \{1,
  2,\dots,n-1\}$, the number of chains $\cC \subseteq P \sm \{ \hat{0},
  \hat{1} \}$ such that $\{\rho(x): x \in \cC\} = S$ will be denoted by $f_P
  (S)$. For instance, $f_P (S)$ is equal to the number of elements of $P$ of
  rank $k$, if $S = \{k\} \subseteq [n-1]$, and to the number of maximal
  chains of $P$, if $S = [n-1]$. The function which maps $S$ to $f_P (S)$
  for every $S \subseteq [n-1]$ is an important enumerative invariant of $P$,
  known as the \textit{flag $f$-vector}; see, for instance, \cite{BH}.

  The present note is partly motivated by the results of \cite{BE, EK}.
  There it is proven that the Boolean algebra of rank $n$ minimizes the
  {\bf~cd}-index, an invariant which refines the flag
  $f$-vector, among all face lattices of convex polytopes and, more
  generally, Gorenstein* lattices, of rank $n$. It is natural to consider
  lattices which are not necessarily Eulerian, in this context. To state
  our main result, we fix some more notation as follows. We denote by $\Delta
  (Q)$ the simplicial complex consisting of all chains in a finite poset $Q$,
  known as the \emph{order complex} \cite{Bj} of $Q$, and by $\widetilde{H}_*
  (\Delta; \kk)$ the reduced simplicial homology over $\kk$ of an abstract
  simplicial complex $\Delta$, where $\kk$ is a fixed field or $\ZZ$. We
  denote by $B_n$ the Boolean algebra of rank $n$ (meaning, the lattice of
  subsets of the set $[n]$, partially ordered by inclusion) and recall that
  if $S = \{s_1 < s_2 < \cdots < s_l \} \subseteq [n-1]$, then $f_{B_n} (S)$
  is equal to the multinomial coefficient $\alpha_n (S) = {n \choose s_1,
  s_2-s_1,\dots,n-s_l}$.

  \begin{theorem} \label{thm:flag-f}
     Let $L$ be a finite graded lattice of rank $n$, with minimum element
     $\hat{0}$ and maximum element $\hat{1}$, and let $\bar{L} = L \sm \{
     \hat{0}, \hat{1} \}$ be the proper part of $L$. If $\widetilde{H}_{n-2}
     (\Delta(\bar{L}); \kk) \ne 0$, then
       \begin{equation} \label{eq:minf}
         f_L (S) \ge \alpha_n (S)
      \end{equation}
     for every $S \subseteq [n-1]$. In other words, the Boolean algebra of
     rank $n$ minimizes the flag $f$-vector among all finite graded lattices
     of rank $n$ whose proper part has nontrivial top-dimensional reduced
     homology over $\kk$.
  \end{theorem}

  A similar statement, asserting that the Boolean algebra of rank $n$ has
  the smallest number of elements among all finite lattices $L$ satisfying
  $\widetilde{H}_{n-2} (\Delta(\bar{L}); \ZZ) \ne 0$, was proved by
  Meshulam \cite{Me}. The proof of Theorem \ref{thm:flag-f}, given in
  Section \ref{sec:fproof}, is elementary and similar in spirit to (but
  somewhat more involved than) the proof of the result of \cite{Me}. A
  different (but less elementary) proof may be given using the methods of
  \cite[Section 2]{EK}. In the remainder of this section we discuss some
  consequences of Theorem \ref{thm:flag-f} and a related open problem.

  The \emph{$f$-vector} of a simplicial complex $\Delta$ is defined as the
  sequence $f(\Delta) = (f_0, f_1,\dots)$, where $f_i$ is the number of
  $i$-dimensional faces of $\Delta$. We recall that the order complex
  $\Delta(\bar{B}_n)$ is isomorphic to the barycentric subdivision of the
  $(n-1)$-dimensional simplex. The next statement follows from this
  observation, Theorem \ref{thm:flag-f} and the fact (see, for instance,
  \cite[p.~95]{StaCCA}) that each entry of the $f$-vector of the order
  complex $\Delta(\bar{L})$ can be expressed as a sum of entries of the flag
  $f$-vector of $L$.

  \begin{corollary} \label{cor:f}
     The barycentric subdivision of the $(n-1)$-dimensional simplex has
     the smallest possible $f$-vector among all order complexes
     of the form $\Delta(\bar{L})$, where $L$ is a finite graded lattice of
     rank $n$ satisfying $\widetilde{H}_{n-2} (\Delta(\bar{L}); \kk) \ne 0$.
  \end{corollary}

  Analogous results for the class of flag simplicial complexes have
  appeared in \cite{Ath, GKN, Me0, Ne}.

  Let $P$ be a graded poset of rank $n$, as in the beginning of this section.
  The \emph{flag $h$-vector} of $P$ is the function assigning to each $S
  \subseteq [n-1]$ the integer
    \begin{equation} \label{eq:flag-h}
      h_P (S) \ = \ \sum_{T \subseteq S} \ (-1)^{|S \sm T|} \ f_P (T).
    \end{equation}
  Equivalently, we have
    \begin{equation} \label{eq:flag-fh}
      f_P (S) \ = \ \sum_{T \subseteq S} \ h_P (T)
    \end{equation}
  for every $S \subseteq [n-1]$. We write $\beta_n (S)$ for the entry
  $h_{B_n} (S)$ of the flag $h$-vector of the Boolean algebra of rank $n$
  and recall \cite[Corollary 3.12.2]{StaEC1} that $\beta_n (S)$ is equal to
  the number of permutations of $[n]$ with descent set $S$.

  It is known that if $P$ is Cohen-Macaulay over $\kk$ (see \cite[Section
  11]{Bj} or \cite[Section 3.8]{StaEC1} for the definition), then $h_P (S)
  \ge 0$ for every $S \subseteq [n-1]$. Moreover, in this case
  $\Delta(\bar{L})$ has nontrivial top-dimensional reduced homology over
  $\kk$ if and only if $\mu_P (\hat{0}, \hat{1}) \ne 0$, where $\mu_P$ is
  the M\"obius function of $P$. Hence, Theorem \ref{thm:flag-f} implies that
  the Boolean algebra of rank $n$ minimizes the flag $f$-vector among all
  Cohen-Macaulay lattices of rank $n$ with nonzero M\"obius number. In view
  of (\ref{eq:flag-fh}), the following conjecture provides a natural
  strengthening of this statement.

  \begin{conjecture} \label{conj:flag-h}
     Let $L$ be a finite lattice of rank $n$, with minimum element $\hat{0}$
     and maximum element $\hat{1}$. If $L$ is Cohen-Macaulay over $\kk$ and
     $\mu_L (\hat{0}, \hat{1}) \ne 0$, then
       \begin{equation} \label{eq:minh}
         h_L (S) \ge \beta_n (S)
      \end{equation}
     for every $S \subseteq [n-1]$. In other words, the Boolean algebra of
     rank $n$ minimizes the flag $h$-vector among all Cohen-Macaulay
     lattices of rank $n$ with nonzero M\"obius number.
  \end{conjecture}

  This conjecture was initially stated by the author under the assumption
  that $\mu_L (x, y) \ne 0$ holds for all $x, y \in L$ with $x \le_L y$ and
  took its present form after a question raised by R.~Stanley \cite{StaPC},
  asking whether this condition could be relaxed to $\mu_L (\hat{0}, \hat{1})
  \ne 0$. It would imply that among all Cohen-Macaulay order complexes of
  the form $\Delta(\bar{L})$, where $L$ is a lattice of rank $n$
  satisfying $\mu_L (\hat{0}, \hat{1}) \ne 0$, the barycentric subdivision
  of the $(n-1)$-dimensional simplex has the smallest possible $h$-vector
  (the entries of the $h$-vector of this subdivision are the Eulerian
  numbers, counting permutations of the set $[n]$ by the number of descents).
  Conjecture \ref{conj:flag-h} is known to hold for Gorenstein* lattices
  (in this case it follows from the stronger result \cite[Corollary 1.3]{EK},
  mentioned earlier, on the {\bf~cd}-index of such a lattice) and for
  geometric lattices \cite[Proposition 7.4]{BER}.

  \section{Proof of Theorem \ref{thm:flag-f}}
  \label{sec:fproof}

  Throughout this section, $L$ is a lattice as in Theorem \ref{thm:flag-f}.
  For $a, b \in L$ with $a \le_L b$, we denote by $\Delta(a, b)$ (respectively,
  by $\Delta(a, b]$) the order complex of the open interval $(a, b)$
  (respectively, half-open interval $(a, b]$) in $L$. We say that an element
  $x \in L$ is \emph{good} if $x = \hat{0}$ or $\widetilde{H}_{k-2} (\Delta
  (\hat{0}, x); \kk) \ne 0$, where $k$ is the rank of $x$ in $L$, and otherwise
  that $x$ is \emph{bad}.

  \medskip
  The proof of Theorem \ref{thm:flag-f} will follow from the next proposition.

  \begin{proposition} \label{prop:f}
     Under the assumptions of Theorem \ref{thm:flag-f}, the lattice $L$ has at least
     ${n \choose k}$ good elements of rank $k$ for every $k \in \{0, 1,\dots,n\}$.
  \end{proposition}
  \begin{proof}
  We proceed in several steps.

  \medskip
  \noindent {\sf Step 1:} We show that $L$ has at least one good coatom. Suppose, by
  the way of contradiction, that no such coatom exists. Suppose further that $L$ has
  the minimum possible number of coatoms among all lattices of rank $n$ which satisfy
  the assumptions of Theorem \ref{thm:flag-f} and have no good coatom. Since
  $\Delta(\bar{L})$ is non-acyclic over $\kk$, the order complex $\Delta(\bar{L})$
  cannot be a cone and hence $L$ must have at least two coatoms. Let $c$ be one of
  them and consider the complexes $\Delta(\bar{L} \sm \{c\})$ and $\Delta(\hat{0}, c]$.
  The union of these complexes is equal to $\Delta(\bar{L})$ and their intersection is
  equal to $\Delta(\hat{0}, c)$. Since $\Delta(\hat{0}, c]$ is a cone, hence
  contractible, and since $\widetilde{H}_{n-3} (\Delta(\hat{0}, c); \kk) = 0$ by
  assumption, it follows from the Mayer-Vietoris long exact sequence in homology for
  $\Delta(\bar{L} \sm \{c\})$ and $\Delta (\hat{0}, c]$ that
     \begin{equation} \label{eq:L-c}
       \widetilde{H}_{n-2} (\Delta(\bar{L} \sm \{c\}); \kk) \cong
         \widetilde{H}_{n-2} (\Delta(\bar{L}); \kk) \ne 0.
     \end{equation}
  Since $L \sm \{c\}$ may not be graded, we consider the subposet $M = J \cup \{
  \hat{1} \}$ of $L$, where $J$ stands for the order ideal of $L$ generated by all
  coatoms other than $c$. The poset $M$ is a finite meet-semilattice with a maximum
  element and hence it is a lattice by \cite[Proposition 3.3.1]{StaEC1}. Since $L$
  is graded of rank $n$, so is $M$ and the set of $(n-1)$-element chains of
  $\Delta(\bar{M})$ coincides with that of $\Delta(\bar{L} \sm \{c\})$, where $\bar{M}
  = M \sm \{ \hat{0}, \hat{1} \}$ is the proper part of $M$. The last statement and
  (\ref{eq:L-c}) imply that
    $$  \widetilde{H}_{n-2} (\Delta(\bar{M}); \kk) \cong \widetilde{H}_{n-2}
        (\Delta(\bar{L} \sm \{c\}); \kk) \ne 0. $$
  Clearly, all coatoms of $M$ are bad. Since $M$ has one coatom less than $L$, we
  have arrived at the desired contradiction.

  \medskip
  \noindent {\sf Step 2:} Assume that $n \ge 2$ and let $b$ be any coatom of $L$.
  We show that there exists an atom $a$ of $L$ which is not comparable to $b$
  and satisfies $\widetilde{H}_{n-3} (\Delta (a, \hat{1}); \kk) \ne 0$. Arguing by
  contradiction, once again, suppose that no such atom exists. Suppose further that
  the number of atoms of $L$ which do not belong to the interval $[\hat{0}, b]$ is
  as small as possible for a graded lattice $L$ of rank $n$ and coatom $b$ which
  have this property and satisfy the assumptions of Theorem \ref{thm:flag-f}. Since
  $\Delta (\bar{L})$
  is non-acyclic over $\kk$, the Crosscut Theorem of Rota \cite[Theorem 10.8]{Bj}
  implies that there exists at least one atom of $L$ which does not belong to the
  interval $[\hat{0}, b]$. Let $a$ be any such atom and let $M$ be the subposet of
  $L$ consisting of $\hat{0}$ and the elements of the dual order ideal of $L$
  generated by the atoms other than $a$. The arguments in Step 1, applied to the
  dual of $L$, show that $M$ is a graded lattice of rank $n$ which satisfies
  $\widetilde{H}_{n-2} (\Delta(\bar{M}); \kk) \cong \widetilde{H}_{n-2} (\Delta
  (\bar{L}); \kk) \ne 0$. Since $M \sm (\hat{0}, b]$ has one atom less than $L \sm
  (\hat{0}, b]$, this contradicts our assumptions on $L$ and $b$.

  \medskip
  \noindent {\sf Step 3:} We now show that $L$ has at least $n$ good coatoms by
  induction on $n$. The statement is trivial for $n=1$, so suppose that $n \ge 2$.
  By replacing $L \sm \{ \hat{1} \}$ with its order ideal generated by the good
  coatoms, as in Step 1, we may assume that all coatoms of $L$ are good. Let $b$
  be any coatom of $L$. By Step 2, there exists an atom $a$ of $L$ which is not
  comparable to $b$ and satisfies $\widetilde{H}_{n-3} (\Delta (a, \hat{1}); \kk)
  \ne 0$. The interval $[a, \hat{1}]$ in $L$ is a graded lattice of rank $n-1$ to
  which the induction hypothesis applies. Therefore, it has at least $n-1$ coatoms
  and all of these are different from $b$. It follows that $L$ has at least $n$
  coatoms, all of which are good.

  \medskip
  \noindent {\sf Step 4:} We prove the following: Given any integers $0 \le r \le
  k \le n$ and any order ideal $I$ of $L \sm \{ \hat{1} \}$ generated by at most
  $r$ elements, there exist at least ${n-r \choose k-r}$ good elements of $L$ of
  rank $k$ which do not belong to $I$. The special case $r=0$ of this statement,
  in which $I$ is the empty ideal, is equivalent to the proposition. Thus, it
  suffices to prove the statement.

  We proceed by induction on $n$ and $n-r$, in this order. The statement is trivial
  for $n = 1$ and for $r=n$, so we assume that $n \ge 2$ and $0 \le r \le n-1$.
  Consider an order ideal $I$ of $L \sm \{ \hat{1} \}$ generated by at most $r$
  elements and let $k$ be an integer in the range $r \le k \le n$. Since $I$ contains
  at most $r \le n-1$ coatoms of $L$, Step 3 imples that there exists a good coatom,
  say $b$, of $L$ which does not belong to $I$. The interval $[\hat{0}, b]$ of $L$
  is a graded lattice of rank $n-1$ whose proper part has nontrivial top-dimensional
  reduced homology over $\kk$. Moreover, the intersection $I \cap [\hat{0}, b]$ is
  an order ideal of $[\hat{0}, b)$ which is generated by at most $r$ elements, namely
  the meets of $b$ with the maximal elements of $I$. By our induction on $n$, there
  exist at least ${n-r-1 \choose k-r}$ good elements of $[\hat{0}, b]$ of rank $k$
  which do not belong to $I$. The union $J = I \cup [\hat{0}, b]$ is an order ideal
  of $L \sm \{ \hat{1} \}$ which is generated by at most $r+1$ elements. By our
  induction on $n-r$, there exist at least ${n-r-1 \choose k-r-1}$ good elements of
  $L$ of rank $k$ which do not belong to $J$. We conclude that there exist at least
  ${n-r-1 \choose k-r} + {n-r-1 \choose k-r-1} = {n-r \choose k-r}$ good elements of
  $L$ of rank $k$ which do not belong to $I$. This completes the inductive step and
  the proof of the statement.
  \end{proof}

  \noindent
  \emph{Proof of Theorem \ref{thm:flag-f}}. We proceed by induction on $n$. The
  result is trivial for $n = 1$ and for $S = \varnothing$, so we assume that $n \ge
  2$ and choose a nonempty subset $S$ of $[n-1]$. We denote by $k$ the largest element
  of $S$ and observe that $f_L (S)$ is equal to the number of pairs $(x, \cC)$, where
  $x$ is an element of $L$ of rank $k$ and $\cC$ is a chain in the interval $[\hat{0},
  x]$, such that the set of ranks of the elements of $\cC$ is equal to $S \sm \{k\}$.
  By Proposition \ref{prop:f}, there are at least ${n \choose k}$ good elements $x$
  of rank $k$ in $L$ and each of the intervals $[\hat{0}, x]$ is a graded lattice of
  rank $k$ whose proper part has nontrivial top-dimensional reduced homology over
  $\kk$. Thus, the induction hypothesis applies to these intervals and we may
  conclude that
    $$ f_L (S) \, \ge \, {n \choose k} \alpha_k (S \sm \{k\}) \, = \, \alpha_n (S). $$
  This completes the induction and the proof of the theorem.
  \qed

  \medskip
  We end with a note on the case of equality in (\ref{eq:minf}). It was shown
  in \cite{Me} that every lattice $L$ which satisfies $\widetilde{H}_{n-2}
  (\Delta(\bar{L}); \ZZ) \ne 0$ and has cardinality $2^n$ must be isomorphic to
  the Boolean algebra $B_n$. As a result, if equality holds in (\ref{eq:minf})
  for every singleton $S \subseteq [n-1]$, then $L$ is isomorphic to $B_n$. Using
  the arguments in this section, as well as induction on $n$ and $k$, the following 
  statement has been verified by Kolins and Klee \cite{KK}: if $L$ satisfies the 
  assumptions of Theorem \ref{thm:flag-f} and for some $k \in \{1, 2,\dots,n-1\}$ 
  equality holds in (\ref{eq:minf}) for every subset $S$ of $[n-1]$ of cardinality $k$, 
  then $L$ is isomorphic to the Boolean algebra of rank $n$.

  \end{document}